\documentclass[11pt]{amsart}
\usepackage{amsmath,amssymb,amsthm,hyperref}
\usepackage{graphicx}
\usepackage{tikz}

\newtheorem{thm}{Theorem}
\newtheorem{lem}[thm]{Lemma}

\newtheorem{prp}[thm]{Proposition}
\theoremstyle{remark}
\newtheorem{rem}[thm]{Remark}
\theoremstyle{definition}

\newcommand{\R}{\mathbb{R}}
\newcommand{\E}{\mathbb{E}}

\begin{document}

\title
[Limit theorems for  
Hermite-driven processes]
{Limit theorems for integral functionals
of Hermite-driven processes}

\author{Valentin Garino}
\address{Universit\'{e} du Luxembourg, D\'epartement de math\'{e}matiques}
\email{valentin.garino@uni.lu}
\thanks{}

\author{Ivan Nourdin}
\address{Universit\'{e} du Luxembourg, D\'epartement de math\'{e}matiques}
\email{ivan.nourdin@uni.lu}
\thanks{}

\author{David Nualart}
\address{Kansas University}
\email{nualart@ku.edu}
\thanks{}

\author{Majid Salamat}
\address{Isfahan university of Technology, Department of Mathematical sciences}
\email{msalamat@iut.ac.ir}
\thanks{}


\begin{abstract}
Consider a moving average process $X$ of the form $X(t)=\int_{-\infty}^t x(t-u)dZ_u$, $t\geq 0$, where
$Z$ is a (non Gaussian) Hermite process of order $q\geq 2$ and $x:\mathbb{R}_+\to\mathbb{R}$ is sufficiently integrable.
This paper investigates the fluctuations, as $T\to\infty$, of integral functionals of the form
$t\mapsto \int_0^{Tt }P(X(s))ds$, in the case where $P$ is any given polynomial function.
It extends a study initiated in \cite{Tran}, where only the quadratic case $P(x)=x^2$ and the convergence in the sense of finite-dimensional distributions were considered.
\end{abstract}

\maketitle

\allowdisplaybreaks
\section{Introduction}
Hermite processes have become more and more popular in the recent literature, as they occur naturally when we consider limits of partial sums associated with long-range dependent stationary series.
They form a family of stochastic processes, indexed by an integer $q\geq 1$ and a Hurst parameter $H\in(\frac12,1)$,
that contains the fractional Brownian motion ($q=1$) and the Rosenblatt
process ($q=2$) as particular cases.
We refer the reader to Section \ref{sechermite} and the references therein for a precise definition. Of primary importance in the sequel is the
self-similarity  parameter $H_0$, given in terms of $H$ and $q$ by
\begin{equation}\label{hzero}
H_0=1-\frac{1-H}{q}\in(1-\frac1{2q},1).
\end{equation}

The goal of the present paper is to investigate the fluctuations, as $T\to\infty$, of the family of stochastic processes 
\begin{equation}\label{intro}
t\mapsto \int_0^{Tt }P(X(s))ds,\quad t\in [0,1]\,\,\mbox{(say)},
\end{equation}
in the case where $P(x)$ is a polynomial function and $X$ is a moving average process of the form
\begin{equation}\label{X}
X(t)=\int_{-\infty}^t x(t-u)dZ_u,\quad t\geq 0,
\end{equation} 
with $Z$ a Hermite process  and $x:\mathbb{R}_+\to\mathbb{R}$ a sufficiently integrable function. 
We note that integral functionals such as (\ref{intro}) are often encountered in the context of statistical estimation, see e.g.  \cite{Tran} for a concrete example.\\

Let us first consider the case where $q = 1$, that is to say the case where $Z$ is the fractional Brownian motion.
Note that this is the only case where $Z$ is Gaussian, making the study {\it a priori} much simpler and more affordable.
By linearity and passage to the limit, the process $X$ is also Gaussian.
Moreover, it is stationary, since the quantity  $\E[X(t)X(s)]=:\rho(t-s)$ only depends on $t-s$.
For simplicity and without loss of generality, assume that $\rho(0)=1$, that is, $X(t)$ has variance 1 for any $t$.
As is well-known since the eighties (see \cite{BM,DM,Taqqu}), the fluctuations of (\ref{intro}) 
heavily depends on the centered  Hermite rank of $P$, defined as
the integer $d\geq 1$ such that $P$ decomposes in the form 
\begin{equation}\label{hermite}
P =\E[P(X(0))]+ \sum_ {k = d}^\infty a_k H_k,
\end{equation}
with $ H_k $ the $k$th Hermite polynomials and $a_d \neq 0$.
(Note that the sum (\ref{hermite}) is actually finite, since $P$ is a polynomial, so that $\# \{k: \, a_k \neq 0 \} <\infty $.)

The first result of this paper concerns the fractional Brownian motion. 
Even if it does not follow directly from the well-known results of Breuer-Major \cite{BM}, Dobrushin-Major \cite{DM} and Taqqu \cite{Taqqu}, the limits obtained are somehow expected. 
In particular, the threshold $H = 1-\frac1{2d}$ is well known to specialists. 
However, the proof of this result is not straightforward, and requires several estimations  which are interesting in themselves.

\begin{thm}\label{FBM}
Let $Z$ be a fractional Brownian motion of Hurst index $H\in (\frac12, 1)$, and
let $x\in L^1(\R_+)\cap L^{\frac{1}{H}}(\R_+)$.
Consider the moving average process $X$ defined by (\ref{X}) and assume without loss of generality that ${\rm Var}(X(0))=1$
(if not, it suffices to multiply $x$ by a constant).
Finally, let $P(x)=\sum_{n=0}^N a_nx^n$ be a real-valued polynomial function, and let $d\geq 1$ denotes its centered Hermite rank.
\begin{enumerate}
\item If $d\geq 2$ and $H\in(\frac12,1-\frac1{2d})$ then
\begin{equation}\label{cv1}
T^{-\frac12}\left\{\int_0^{Tt} \big(P(X(s))-\mathbb{E}[P(X(s))]\big)ds\right\}_{t\in [0,1]}
\end{equation}
converges in distribution in $C([0,1])$ to 
a standard Brownian motion $W$, up to some multiplicative constant $C_1$ which is explicit and depends only on $x$, $P$ and $H$.
\item If  $H\in(1-\frac1{2d},1)$ then
\begin{equation}\label{cv2}
T^{d(1-H)-1}\left\{\int_0^{Tt} \big(P(X(s))-\mathbb{E}[P(X(s))]\big)ds\right\}_{t\in [0,1]}
\end{equation}
converges in distribution in $C([0,1])$ to 
a Hermite process of index $d$ and Hurst parameter $1-d(1-H)$, up to some multiplicative constant $C_2$ which is explicit and depends only on $x$, $P$ and $H$.
\end{enumerate}
\end{thm}

Now, let us consider the non-Gaussian case, that is, the case where $q\geq 2$.
As we will see, the situation is completely different, both in the results obtained (rather  unexpected) and in the methods used
(very different from the Gaussian case).
Let $L>0$. We define $\mathcal S_L$ to be the set of bounded functions $l:\mathbb{R_+}\rightarrow\mathbb{R}$ such that $y^L l(y)\to 0$ as $y\to\infty$.
We observe that  $\mathcal S_L\subset  L^1(\mathbb{R_+})\cap L^{\frac1H}(\mathbb{R_+})$ for any $L>1$.
We can now state the following result.
\begin{thm}\label{main}
Let $Z$ be a Hermite process of order $q\geq 2$ and Hurst parameter $H\in (\frac{1}{2}, 1)$, and
let $x\in \mathcal S_L$ for some $L>1$.
Recall $H_0$ from (\ref{hzero}) and consider the moving average process $X$ defined by (\ref{X}).
Finally, let $P(x)=\sum_{n=0}^N a_nx^n$ be a real-valued polynomial function. Then, one and only one of the following two situations takes place at $T\to\infty$:
\begin{enumerate}
\item[(i)] If $q$ is odd and if $a_n\neq 0$ for at least one odd $n\in\{1,\ldots,N\}$, then
$$
T^{-H_0}\left\{\int_0^{Tt} \big(P(X(s))-\mathbb{E}[P(X(s))]\big)ds\right\}_{t\in [0,1]}
$$
converges in distribution in $C([0,1])$ to a fractional Brownian motion of parameter $H_1:=H_0$, up to some multiplicative constant $K_1$ which is explicit and depends only on $x$, $P$, $q$ and $H$, see Remark \ref{rk}.
\item[(ii)] If $q$ is even, or if $q$ is odd and $a_n=0$ for all odd $n\in\{1,\ldots,N\}$, then
$$
T^{1-2H_0}\left\{\int_0^{Tt} \big(P(X(s))-\mathbb{E}[P(X(s))]\big)ds\right\}_{t\in [0,1]}
$$
converges in distribution in $C([0,1])$ to a Rosenblatt process  of Hurst parameter $H_2:=2H_0-1$, up to some multiplicative  constant $K_2$ which is explicit and depends only on $x$, $P$, $q$ and $H$, see Remark \ref{rk}.
\end{enumerate}
\end{thm}

\begin{rem}\label{rk}
{\rm
Whether in Theorem \ref{FBM} or Theorem \ref{main}, the multiplicative constants appearing in the limit can be all given explicitly by following the respective proofs. For example,
the constant $K_1$ and $K_2$ of Theorem \ref{main} are given by the following intricate expressions:
\begin{eqnarray*}
K_{1}&=&\sum_{n= 3, n\,  {\rm odd}}^Na_{n} c_{H,q}^nK_{x,n,1} \\
K_{2}&=&\sum_{n=2}^N a_{n} c_{H,q}^nK_{x,n,2}+a_1\mathbb{I}_{\{q=2\}}\int_{\R_+} x(\nu) d\nu 
\end{eqnarray*}
with
\[
K_{x,n,i}= \sum_{\alpha \in A_{n,q}, nq -2 |\alpha|=i } \frac {C_\alpha  K_{x,\alpha, H_0}} { c_{H_i, i}}, \quad i=1,2,
\]
where the sets and constants in the previous formula are defined in Sections \ref{preli}, \ref{ChaoticExp} and \ref{sec5}.
}
\end{rem}

We note that our Theorem \ref{main} contains as a very particular case the main result of \cite{Tran}, which corresponds to the choice $P(x)=x^2$ and 
thus situation (ii). Moreover, let us emphasize that our Theorem \ref{main} not only studies the convergence of finite-dimensional distributions as in \cite{Tran}, but also provides a {\it functional} result.

Because the employed method  is new,
let us sketch the main steps of the proof of Theorem \ref{main}, by using the
classical notation of the Malliavin calculus (see Section 2 for any unexplained definition or result); in particular we write $I_p^B(h)$ to indicate the $p$th multiple Wiener-It\^o integral of kernel $h$ with respect to the standard (two-sided) Brownian motion $B$.
 
 \medskip
{\bf (Step 1)}  In Section \ref{ChaoticExp}, we represent the moving average process $X$ as a $q$th multiple Wiener-It\^o integral with respect to $B$:
$$
X(t)=c_{H,q}\,I^B_q(g(t,\cdot)),
$$ 
where  $c_{H,q}$ is an explicit constant and the kernel $g(t,\cdot)$ is given by
\begin{equation}\label{gg}
g(t, \xi_1, \dots, \xi_q)=\int_{-\infty}^tx{(t-v)}\prod_{j=1}^q(v-\xi_j)_+^{H_0-\frac32}dv,  
\end{equation}
for $ \xi_1, \dots, \xi_q\in\mathbb{R}$, $ t\geq 0$.

 \medskip
{\bf (Step 2)}
In Lemma \ref{Un} we  compute the chaotic expansion of the $n$th power of $X(t)$ for any $n\geq 2$ and $t>0$, and obtain an expression of the form
$$
X^n(t)=
c^n_{H, q}\sum_{  \alpha \in A_{n,q}}C_\alpha I_{nq-2|\alpha|}^B (\otimes _\alpha ( g(t, \cdot), \dots, g(t, \cdot))),
$$
where we have used the novel notation $ \otimes _\alpha ( g(t, \cdot), \dots, g(t, \cdot))$ to indicate iterated contractions whose precise definition is given in Section \ref{label},
and where $C_\alpha$ are combinatorial constants and the sum runs over a family $A_{n,q}$ of suitable multi-indices $ \alpha
=( \alpha_{ij}, 1\le i<j\le n)$. As an immediate consequence, we deduce that our quantity of interest can be decomposed as follows:
\begin{align}\label{reprintro}
&\int_0^{Tt} \big(P(X(s))-\mathbb{E}[P(X(s))]\big)ds =a_0\int_0^{Tt} X(s)ds\\
& \hskip1cm+\sum_{n=2}^N a_n \, c_{H,q}^n \notag
\sum_{  \alpha \in A_{n,q}, nq -2 |\alpha| \ge 1}C_\alpha \int_{0}^{Tt} I_{nq-2 |\alpha|}^B (\otimes _\alpha ( g(s, \cdot), \dots, g(s, \cdot)))ds.
\end{align}

 \medskip
{\bf (Step 3)} In Proposition \ref{prop2}, we compute an explicit expression for the iterated contractions  $ \otimes _\alpha ( g(t, \cdot), \dots, g(t, \cdot))$ appearing in the right-hand side of (\ref{reprintro}), by using that $g$ is given by (\ref{gg}).\\
 
To ease the description of the remaining steps, let us now set 
\begin{equation}\label{cf}
F_{n,q,\alpha,T}(t)
=
 \int_0^{Tt} I^B_{nq-2|\alpha|}( \otimes _\alpha ( g(s, \cdot), \dots, g(s, \cdot)))ds.
\end{equation}
 
 \medskip
{\bf (Step 4)}
As $T\to \infty$, we show in 
Proposition  \ref{assympofF}
that, if  $nq-2|\alpha|<\frac{q}{1-H}$, then   $T^{-1+(1-H_0)(nq-2|\alpha|)}F_{n,q,\alpha,T}(t)$  converges in distribution to a Hermite process (whose order and Hurst index are specified) up to some multiplicative constant. Similarly, we prove in Proposition  \ref{prp5} that, if $nq -2|\alpha| \ge 3$, then $T^{\alpha_0}F_{n,q,\alpha,T}(t)$  is tight and converges   in $L^2(\Omega)$ to zero, where $\alpha_0$ is given in (\ref{alpha0}).
 
  \medskip
{\bf (Step 5)}
By putting together the results obtained in the previous steps, the two convergences stated in Theorem \ref{main} follow immediately.\\
 
To illustrate a possible use of our results,
we study in Section \ref{sec5bis} an extension of the classical fractional Ornstein-Uhlenbeck process (see, e.g., Cheridito {\it et al} \cite{cheredito2003}) to the case where the driving process is more generally a Hermite process. To the best of our knowledge, 
there is very little literature devoted to this mathematical object, only \cite{nourdin2019statistical,slaoui2018limit}. \\

The rest of the paper is organized as follows. 
Section \ref{preli} presents some basic results about 
multiple Wiener-It\^o integrals and Hermite processes,
as well as some other facts that are
used throughout the paper. 
Section \ref{ChaoticExp} contains preliminary results. 
The proof of Theorem \ref{FBM} (resp. Theorem \ref{main}) is given in Section \ref{sec4} (resp. Section \ref{sec5}). 
In Section \ref{sec5bis}, we provide a complete asymptotic study of the Hermite-Ornstein-Uhlenbeck process, by means of  Theorems \ref{FBM} and  \ref{main} and of an extension of Birkhoff's ergodic Theorem. 
Finally,  Section \ref{sec6} contains two technical  results: a power counting theorem and  a version of the Hardy-Littelwood inequality, which both play an important role in the proof of our main theorems.

\section{Preliminaries on multiple Wiener-It\^o integrals and Hermite processes}\label{preli}

\subsection{Multiple Wiener-It\^o integrals and a product formula}\label{label}
A function $f:\mathbb{R}^p\to\mathbb{R}$ is said to be {\it symmetric} if the following relation holds for all permutation $\sigma\in \mathfrak{S}(p)$:
$$
f(t_1, \dots, t_p)=f(t_{\sigma(1)}, \dots, t_{\sigma(p)}),\quad 
t_1, \dots, t_p\in\mathbb{R}.
$$
The subset of $L^2(\mathbb{R}^p)$ composed of symmetric functions is denoted by $L_s^2(\mathbb{R}^p)$.

Let $B=\{B(t)\}_{t\in\mathbb{R}}$ be a two-sided Brownian motion. For any given $f\in L^2_s(\mathbb{R}^p)$ we consider the  {\it multiple Wiener-It\^{o} integral} of $f$ with respect to $B$, denoted by
$$
I^B_p(f) =\int_{\mathbb{R}^p}f(t_1, \dots, t_p) dB(t_1)\cdots dB(t_p).
$$
This stochastic integral satisfies $\mathbb{E}[I^B_p(f)]=0$ and 
\[
\mathbb{E}[I^B_p(f)I^B_q(g)]=\mathbf{1}_{\{p=q\}} p! \langle f, g\rangle_{L^2(\mathbb{R}^p)}
\] 
for $f\in L^2_s(\R^p)$ and $g\in L^2_s(\R^q)$, 
see \cite{NP1} and \cite{Nualart} for precise definitions and further details.

It will be convenient in this paper to deal with multiple Wiener-It\^o integrals of possibly nonsymmetric functions.
If $f\in  L^2(\R^p)$, we put  $I_p^B(f)= I_p^B({\tilde f})$, where $\tilde{f}$ denotes the symmetrization of $f$, that is,
\[
\tilde{f}(x_1, \dots, x_p) =\frac 1 {p!} \sum_{\sigma\in  \mathfrak{S}(p)} f(x_{\sigma(1)}, \dots, x_{\sigma(p)}).
\]

 We will need  the expansion as a sum of multiple Wiener-It\^o integrals for  a product of the form $$\prod_{k=1}^n I_q^B(h_k),$$ where $q\ge 2$ is fixed and  the functions $h_k$ belong to $L^2_s(\R^q)$ for $k=1,\dots, n$.  In order to present this extension of the product formula and to  define the relevant contractions between the functions $h_i$ and $h_j$ that will naturally appear, we introduce some further notation. 
   Let 
$A_{n,q}$ be the set of multi-indices $\alpha= (\alpha_{ij}, 1\le i<j \le n)$ such that, for each $k=1,\dots, n$, 
\[
\sum_{1\leq i<j\leq n}  \alpha_{ij} {\bf 1}_{k\in\{i,j\}}    \le q.
\]
Set $|\alpha| = \sum_{1\le i<j \le n} \alpha_{ij}$, 
\[
\beta^0_k = q-\sum_{1\leq i<j\leq n}  \alpha_{ij} {\bf 1}_{k\in\{i,j\}}    , \qquad 1\le k\le n
\]
and
\begin{equation} \label{em}
m:= m(\alpha)= \sum_{k=1}^n \beta^0_k  = nq-2|\alpha|.
\end{equation}
For each $ 1\le i < j \le n$, the integer $\alpha_{ij}$ will represent the number of variables in  $h_i$ which are contracted with $h_j$ whereas, for each $k=1,\dots, n$, the integer $\beta^0_k$ is the number of variables in $h_k$ which are not contracted.  We will also write
$\beta_k = \sum_{j=1}^k \beta^0_j$ for $k=1,\dots, n$ and $\beta_0=0$.
Finally, we set
\begin{equation} \label{Calpha}
C_\alpha =\frac {q!^n}{  \prod _{k=1}^n \beta^0_k!  \prod_{ 1\le i  <j\le n}  \alpha_{ij}! }.
\end{equation}
With these  preliminaries,    for any element $\alpha \in A_{n,q}$ we can define the contraction $\otimes_\alpha(h_1, \dots, h_n)$ as the function of $nq-2|\alpha|$ variables obtained by contracting $\alpha_{ij}$ variables between $h_i$ and $h_j$ for each couple of indices $1\le i<j \le n$. Define the collection $(u^{i,j})_{1\leq i,j\leq n,i\neq j}$ in the following way:
$$u^{i,j}=\alpha_{\min(i,j),\max(i,j)}.$$
We then have 
\begin{align}   \notag
&\otimes_\alpha(h_1, \dots, h_n)(\xi_1, \dots, \xi_{nq-2|\alpha|}) \\
& =\int_{\mathbb{R}^{|\alpha|}}  
 \prod_{k=1}^n h_k( s^{k,1}_1,\dots,s^{k,1}_{u^{k,1}},\dots,s^{k,n}_{1},\dots,s^{k,n}_{u^{k,n}} ,\xi_{1+\beta_{k-1}}, \dots, \xi_{\beta_k} )\\
&\qquad\times\prod_{1\leq i<j\leq n}ds^{i,j}_{1}\dots ds^{i,j}_{u^{i,j}} \notag
\end{align}\label{k4}
Notice that the function  $\otimes_\alpha(h_1, \dots, h_n)$ is not necessarily symmetric.

Then, we have the following result.

\begin{lem}\label{Un}
Let $n,q\geq 2$ be some integers and let $h_i\in L^2_s(\mathbb{R}^q)$ for $i=1,\dots, n$.
We have
\begin{align}\label{Un1}
\prod_{k=1}^n I_q^B(h_k)=\sum_{ \alpha \in A_{n,q}} C_\alpha  I^B_{nq-2|\alpha|}( \otimes _\alpha(h_1, \dots, h
_n) ).
\end{align}
\end{lem}

 \begin{proof}
The product formula for multiple stochastic integrals (see, for instance,  \cite[Theorem 6.1.1]{PT}, or  formula (2.1) in \cite{BN}  for $n=2$)   says that
\begin{equation}  \label{product}
 \prod _{k=1}^n I^B_{q}( h_k)  
= \sum_{  \mathcal{P}, \psi }    I^B_{\beta_1^0 + \cdots  +\beta^0_n} \left(   \left( \otimes_{k=1}^n h_k\right)_{\mathcal{P}, \psi} \right),
\end{equation}  
where  $\mathcal{P}$ denotes the set of all partitions $\{1,\dots, q\} = J_i \cup \left(\cup_{ k=1,\dots, n, k \not=i} I_{ik}\right) $, where for any  $i,j =1,\dots,  n$, $I_{ij}$ and $I_{ji}$ have the same cardinality $\alpha_{ij}$, $\psi_{ij}$ is a bijection between  $ I_{ij} $
 and $I_{ji}$ and $\beta^0_k = | J_k|$. Moreover,  $ \left( \otimes_{k=1}^n h_k\right)_{\mathcal{P}, \psi}$ denotes the contraction of
 the indexes $\ell$ and $\psi_{ij} (\ell)$ for any  $\ell  \in I_{ij}$ and any  $i,j =1\dots, n$.
Then, formula  (\ref{Un1}) follows form (\ref{product}), by just counting the number of partitions, which is
\[
\prod_{k=1}^n \frac {q!}{ \prod_{i \, {\rm or} \, j \not =k} \alpha_{ij}! \beta^0_k!}
\]
and multiplying by the number of bijections, which is  $\prod_{1\le i< j \le n} \alpha_{ij}!$.
\end{proof}

 Notice that when $n=2$, formula (\ref{Un1}) reduces to the well-known formula for the product of two multiple integrals.  That is,
for any two symmetric functions $f\in L_s^2(\mathbb{R}^p)$ and $g\in L_s^2(\mathbb{R}^q)$  we have
$$
I^B_p(f)I^B_q(g)=\sum_{r=0}^{\min(p, q)}r! \binom{p}{r}\binom{q}{r}I^B_{p+q-2r}(f\otimes_rg).
$$
where, for $0\leq r\leq \min (p, q)$,   $f\otimes_r g\in L^2(\mathbb{R}^{p+q-2r})$ denotes the   contraction  of $r$ coordinates between $f$ and $g$.

\subsection{Hermite processes}\label{sechermite}

Fix $q\geq 1$ and $H\in(\frac12,1)$.
The Hermite process of index $q$ and Hurst parameter $H$  can be  represented by means of a multiple Wiener-It\^{o}
integral with respect to  $B$  as follows, see e.g. \cite{Maej-Tudor}:
\begin{equation}\label{defher}
Z^{H,q}(t) = c_{H, q} \int_{\mathbb{R}^q} \int_{[0,t]}\prod_{j=1}^{q}(s-x_j)^{H_0-\frac32}_+dsdB(x_1)\cdots dB(x_q),\quad t\in\R.
\end{equation}
Here, $x_+=\max\{x, 0\}$, the constant $c_{H, q}$ is chosen to ensure that  \newline ${\rm Var}(Z^{H,q}(1)) = 1$, and
$$H_0=1-\frac{1-H}q\in\big(1-\frac1{2q},1\big).$$

When $q=1$, the  process $Z^{H,1} $ is Gaussian and is nothing but the fractional Brownian motion with Hurst parameter $H$. 
For $q\geq 2$, the processes $Z^{H,q} $ are no longer Gaussian: they belong to the $q$th Wiener chaos.
The process $Z^{H,2} $ is known as the Rosenblatt process.  

Let $|\mathcal{H}|$ be the following class of functions:
$$
|\mathcal{H}|=\Big\{f:\mathbb{R}\to \mathbb{R}\Big |\,\, \int_\mathbb{R}\int_\mathbb{R}|f(u)||f(v)||u-v|^{2H-2}dudv<\infty\Big\}.
$$
Maejima and Tudor  \cite{Maej-Tudor} proved that the stochastic integral $\int_\mathbb{R} f(u)dZ^{H,q}(u)$ with respect to the Hermite process $Z^{H,q}$ is well defined when $f$ belongs to $|\mathcal{H}|$.
Moreover, for any order $q\geq 1$, index $H\in(\frac12, 1)$
and function $f\in|\mathcal{H}|$, 
\begin{eqnarray}\label{integral}
&&\int_\mathbb{R} f(u)dZ^{H,q}(u)\\
&=&c_{H, q} \int_{\mathbb{R}^q}\left(\int_{\mathbb{R}}f(u)\prod_{j=1}^{q}(u-\xi_j)_+^{H_0-\frac 32 }du\right)dB(\xi_1) \cdots dB(\xi_q).\notag
\end{eqnarray}
As a consequence of the Hardy-Littlewood-Sobolev inequality featured in \cite{Beckner}, 
we observe that $L^1(\mathbb{R})\cap L^{\frac1H}(\mathbb{R})\subset
|\mathcal{H}|$.

\section{Chaotic decomposition of $\int_0^{Tt }P(X(s))ds$}\label{ChaoticExp}
Assume $x\in|\mathcal{H}|$ and $q\geq 1$. Using \eqref{integral} and bearing in mind the notation and results from Section \ref{preli}, it is immediate that $X$ can be written as
\begin{equation}\label{Xbis}
X(t)=c_{H, q}I^B_q(g(t, .)),
\end{equation}
where $g(t, .)$ is given by
\begin{equation}\label{g-T}
g(t, \xi_1, \dots, \xi_q)=\int_{-\infty}^tx{(t-v)}\prod_{j=1}^q(v-\xi_j)_+^{H_0-3/2}dv,
\end{equation}
and $c_{H, q}$ is defined as in (\ref{defher}).

\subsection{Computing the chaotic expansion of $X(t)^n$ when $n\geq 2$}

 Let us denote by $A_{n,q}^0$ the set of elements $\alpha \in  A_{n,q}$ such that $nq-2|\alpha|=0$ and $A_{n,q}^1$ will be  the set of elements $\alpha \in  A_{n,q}$ such that $nq-2|\alpha|\ge 1$.
Notice that   when $nq$ is odd, $A_{n,q}^0$ is empty.
Using \eqref{Un1}, we obtain the following formula for the expectation of the $n$th power ($n\geq 2$) of $X$ given by \eqref{X}:
\begin{equation}\label{expec}
   \mathbb{E}[X(t)^n]=  (c_{H, q})^n\sum_{  \alpha \in A_{n,q}^0}C_\alpha I_{nq-2|\alpha|}^B (\otimes _\alpha ( g(t, \cdot), \dots, g(t, \cdot))).
   \end{equation}
We observe in particular that $\mathbb{E}[X(t)^n]=0$ whenever $nq$ is odd.
From \eqref{Un1} and \eqref{expec}, we deduce for $n\geq 2$ that

\begin{align*}
X(t)^n - \mathbb{E}[X(t)^n]=(c_{H, q})^n\sum_{  \alpha \in A_{n,q}^1}C_\alpha I_{nq-2|\alpha|}^B (\otimes _\alpha ( g(t, \cdot), \dots, g(t, \cdot))),
\end{align*}
implying in turn
\begin{align}\label{repr}
&\int_0^{Tt} \big(P(X(s))-\mathbb{E}[P(X(s))]\big)ds = 
a_1\int_0^{Tt} X(s)ds \\
& \qquad \qquad  + \sum_{n=2}^N a_n 
(c_{H, q})^n\sum_{  \alpha \in A_{n,q}^1}C_\alpha  \int_0^{Tt} I_{nq-2|\alpha|}^B (\otimes _\alpha ( g(s, \cdot), \dots, g(s, \cdot)))ds.
\notag
\end{align}

\subsection{Expressing the iterated   contractions of $g$}
We now compute an explicit expression for the iterated   contractions appearing in (\ref{repr}).

\begin{prp}\label{prop2}
Fix  $n\geq 2$, $q\geq 1$ and $\alpha\in A_{n,q}$.
We have
\begin{align*}
&\otimes _\alpha (g(t, .) , \dots, g(t, .))(\xi)
= \beta(H_0-\frac12,2-2H_0)^{|\alpha |}     \\
&  \qquad \times  \int_{(-\infty,t]^n}dv_1\ldots dv_n\prod_{k=1}^n x(t-v_k)
\prod_{1\leq i<j\leq n}|v_i-v_j|^{(2H_0-2)\alpha_{ij}}  \\
& \qquad   \times \prod_{k=1}^{n}\prod_{\ell=1+\beta_{k-1}}^{\beta_k} (v_k - \xi_{\ell})_+^{H_0-\frac32},
\end{align*}
with the convention $\beta_0=0$.
\end{prp}
\begin{proof}
The proof is a straightforward consequence of the following identity
\begin{equation}\label{bienutile}
\int_{\mathbb{R}}(v-\xi)_+^{H_0-3/2}(w-\xi)_+^{H_0-3/2} d\xi=\beta(H_0-\frac12, 2-2H_0)|v-w|^{(2H_0-2)},
\end{equation}
whose proof is elementary, see e.g. \cite{Biagini}.
\end{proof}

\section{Proof of Theorem \ref{main}}\label{sec5}

We are now ready to prove Theorem \ref{main}.
To do so, we will mostly rely on the forthcoming Proposition \ref{assympofF}, which might be a result of independent interest by itself, and
which studies the asymptotic behavior of $F^B_{n,q,\alpha, T}$ given by (\ref{cf}).
We will denote by f.d.d. \!the convergence in law of the finite-dimensional distributions of a given process. Notice that the hypothesis on $x$ is a bit weaker than the one in the main theorem, the fact that $x\in \mathcal{S}_L$ being required in the forthcoming Proposition \ref{prp5}.

\begin{prp}\label{assympofF}
Fix $n\geq 2$, $q\geq 1$ and $\alpha\in A_{n,q}$. Assume the function $x$ belongs to $L^1(\mathbb{R}_+)\cap L^{\frac1H}(\mathbb{R}_+)$, recall $H_0$ from (\ref {hzero})  and let $m$  be defined as  in  (\ref{em}).  Finally, assume that $2m< \frac{q}{1-H}$ (which is automatically satisfied when $m=1$ or $m=2$). Then, as $T\to\infty$, 
\begin{equation} \label{fdd}
(T^{-1+(1-H_0)m}F_{n,q,\alpha, T} (t))_{t\in [0,1]} \stackrel{f.d.d.}{\longrightarrow}  \left( \frac {C_\alpha K_{x,\alpha,H_0}} { c_{H(m),m}} Z^{H(m),m}(t) \right)_{t\in [0,1]},
\end{equation}
where  $Z^{H(m),m}$ denotes the
 $m$th Hermite process  of Hurst index $H(m)= 1-\frac mq(1-H)$ and the constants $C_\alpha$ and $ K_{x,\alpha,H_0}$ are defined in (\ref{Calpha}) and (\ref{Kalpha}), respectively.
 Furthermore, $\{T^{-1+(1-H_0)m}(F_{n,q,\alpha, T}(t))_{t\in [0,1]},\,T>0\}$ is tight in $C([0,1])$.
\end{prp}

\begin{proof}
Let $n\geq 2$, $q\ge 1$ and $\alpha \in A_{n,q}$.

\medskip
\noindent
\underline{\it Step 1}:  We will first show the convergence  (\ref{fdd}).
We will   make several change of variables in order to transform the expression of $F_{n,q,\alpha, T}(t)$.  By means of an application of stochastic's Fubini's theorem, we can write
\[
F_{n,q,\alpha, T}(t) = C_\alpha \int_{\R^m}    \Psi_T(\xi_1, \dots, \xi_m) dB(\xi_1) \cdots dB(\xi_m),
\]
where
\begin{align*}
\Psi_T(\xi_1, \dots, \xi_m) &:= T^{-1+m(1-H_0)} \int_0^{Tt} ds \int_{(-\infty, s]^n} dv_1 \cdots dv_n \prod_{k=1} ^n x(s-v_k) \\
&\times \prod_{1\le i<j \le n} | v_i-v_j|^{(2H_0-2)\alpha_{ij} }
\prod_{k=1}^n \prod_{\ell = 1+ \beta_{k-1}}^{\beta_k} (v_k -\xi_\ell)_+^{H_0 -\frac 32}.
\end{align*}
Using the change of variables $s\to Ts$ and $v_k \to Ts-v_k$, $1\le k\le n$, we obtain
\begin{align*}
\Psi_T(\xi_1, \dots, \xi_m) &:= T^{m(1-H_0)} \int_0^{t} ds \int_{(-\infty, Ts]^n} dv_1 \cdots dv_n \prod_{k=1} ^n x(Ts-v_k) \\
&\times \prod_{1\le i<j \le n} | v_i-v_j|^{(2H_0-2)\alpha_{ij} }
\prod_{k=1}^n \prod_{\ell = 1+ \beta_{k-1}}^{\beta_k} (v_k -\xi_\ell)_+^{H_0 -\frac 32} \\
&=  T^{-\frac m2} \int_0^{t} ds \int_{[0,\infty)^n} dv_1 \cdots dv_n \prod_{k=1} ^n x(v_k) \\
&\times \prod_{1\le i<j \le n} | v_i-v_j|^{(2H_0-2)\alpha_{ij} }
\prod_{k=1}^n \prod_{\ell = 1+ \beta_{k-1}}^{\beta_k} (s-\frac{v_k}T - \frac{\xi_\ell}T)_+^{H_0 -\frac 32}.
\end{align*}
By the scaling property of the Brownian motion, the processes  
\[
(F_{n,q,\alpha, T}(t))_{ \alpha \in A_{n,q}, 2\le n\le N, t\in [0,1]}
\]
and
\[
(\widehat{F}_{n,q,\alpha, T}(t))_{ \alpha \in A_{n,q}, 2\le n\le N, t\in [0,1]}
\]
have the same  probability distribution, where
\[
\widehat{F}_{n,q,\alpha, T}(t) = C_\alpha \int_{\R^m}    \widehat{\Psi}_T(\xi_1, \dots, \xi_m)dB(\xi_1) \cdots dB(\xi_m) 
\]
\begin{align*}
\widehat{\Psi}_T(\xi_1, \dots, \xi_m) &:=     \int_0^{t} ds \int_{(-\infty, 0]^n} dv_1 \cdots dv_n \prod_{k=1} ^n x(v_k) \\
&\times \prod_{1\le i<j \le n} | v_i-v_j|^{(2H_0-2)\alpha_{ij} }
\prod_{k=1}^n \prod_{\ell = 1+ \beta_{k-1}}^{\beta_k} (s-\frac{v_k}T - \xi_\ell)_+^{H_0 -\frac 32}.
\end{align*}
Set
\[
\widehat{\Psi}(\xi_1, \dots, \xi_m) :=   K_{x,\alpha, H_0}   \int_0^{t} ds    \prod_{\ell = 1 }^m (s  - \xi_\ell)_+^{H_0 -\frac 32},
\]
where
\begin{equation} \label{Kalpha}
K_{x,\alpha, H_0} = \int_{\R_+^n} dv_1 \cdots dv_n \prod_{k=1} ^n x(v_k) \prod_{1\le i<j \le n} | v_i-v_j|^{(2H_0-2)\alpha_{ij} }.
\end{equation}
Notice that, by Lemma \ref{HLS-lemme}, $K_{x,\alpha, H_0} $ is well defined.
We claim that
\begin{equation} \label{k1}
\lim_{T\rightarrow \infty}  \widehat{\Psi}_T =\widehat{\Psi},
\end{equation}
where the convergence holds in $L^2(\R^m)$. This will imply the convergence in $L^2(\Omega)$ of 
$\widehat{F}_{n,q,\alpha, T}(t)$, as $T\rightarrow \infty$ to a Hermite process of order $m$, multiplied by the constant $C_\alpha K_{x,\alpha,H_0}$.

\medskip
\noindent
{\it Proof of (\ref{k1}):}  It suffices to show that the inner products $\langle \widehat{\Psi}_T, \widehat{\Psi}_T \rangle_{L^2(\R^m)}$ 
and $\langle \widehat{\Psi}_T, \widehat{\Psi} \rangle_{L^2(\R^m)}$ converge, as $T\rightarrow \infty$,  to 
\[
\| \widehat{\Psi} \| ^2_{L^2(\R^m)}= K_{x,\alpha, H_0}^2  \beta(H_0 -\frac 12, 2-2H_0)^m  \int_{[0,t]^2} ds ds' |s-s'| ^{(2H_0-2)m},
\]
which is finite because $m<\frac 1{2(1-H_0)} = \frac q{2(1-H)}$. We will show the convergence of $\langle \widehat{\Psi}_T, \widehat{\Psi}_T \rangle_{L^2(\R^m)}$  
and the second term can be handled by the same arguments.
We have
\begin{align*}
\| \widehat{\Psi}_T \|_{L^2(\R^m)}^2&= 
\int_{[0,t]^2} ds ds' \int_{\R_+^{2n}} dv_1 \cdots dv_n  dv'_1 \cdots dv'_n   \\
&\times \prod_{k=1} ^n x(v_k) x(v'_k) 
  \prod_{1\le i<j \le n} | v_i-v_j|^{(2H_0-2)\alpha_{ij} } | v'_i-v'_j|^{(2H_0-2)\alpha_{ij} } \\
 &\times \prod_{k=1}^n   \beta(H_0 -\frac 12, 2-2H_0)^{\beta_k} |s-s' -\frac{v_k-v'_k}T |^{(2H_0 -2)\beta_k}.
\end{align*}
Let us first show that given $w_k \in \R$, $1\le k\le n$, 
\begin{equation} \label{k2}
\lim_{T\rightarrow \infty} \int_{[0,t]^2} ds ds'  \prod_{k=1}^n  |s-s' -\frac{w_k}T |^{(2H_0 -2)\beta_k} =\int_{[0,t]^2} ds ds'    |s-s'   |^{(2H_0 -2)m} 
\end{equation}
and, moreover, 
\begin{equation} \label{k3}
\sup_{w_k\in\R, 1\le k\le n}   \int_{[0,t]^2} ds ds'  \prod_{k=1}^n  |s-s' -{w_k} |^{(2H_0 -2)\beta_k}< \infty.
\end{equation}
By the dominated convergence theorem and using  Lemma \ref{HLS-lemme}, (\ref{k2}) and (\ref{k3}) imply (\ref{k1}).

To show (\ref{k2}), choose $\epsilon$ such that  $ |w_k| /T <\epsilon$, $1\le k\le n$,  for $T$ large enough (depending on the fixed $w_k$'s). Then, we can write
\begin{align*}
& \int_{[0,t]^2} ds ds'  \left|  \prod_{k=1}^n  |s-s' -\frac{w_k}T |^{(2H_0 -2)\beta_k} -    |s-s'   |^{(2H_0 -2)m}  \right| \\
& \quad  \le t \int_{ |\xi| >2\epsilon }  d\xi \left|  \prod_{k=1}^n  |\xi -\frac{w_k}T |^{(2H_0 -2)\beta_k} -    |\xi   |^{(2H_0 -2)m}  \right| \\
& \qquad + 2 t  \sup_{|w_k| <\epsilon}  \int_{ |\xi| \le 2\epsilon }  d\xi  \prod_{k=1}^n  |\xi -w_k |^{(2H_0 -2)\beta_k} \\
& \quad := B_1 + B_2.
\end{align*}
The term $B_1$ tends to zero a $T\rightarrow \infty$, for each $\epsilon>0$.  On the other hand, the term $B_2$ tends to zero as $\epsilon \rightarrow 0$. Indeed,
\[
B_2=
2 t   \epsilon ^{(2H_0-2)m +1}  \sup_{|w_k| <1}  \int_{ |\xi| \le 2 }   d\xi  \prod_{k=1}^n  |\xi -w_k |^{(2H_0 -2)\beta_k}.
\]
Note that the above supremum is finite because the function $(w_1, \dots, w_k) \to \int_{ |\xi| \le 2 }   d\xi  \prod_{k=1}^n  |\xi -w_k |^{(2H_0 -2)\beta_k}$ is continuous. 

Property (\ref{k3}) follows immediately from the fact that the function
\[
(w_1, \dots, w_k) \to  \int_{[0,t]^2} ds ds'  \prod_{k=1}^n  |s-s' -{w_k} |^{(2H_0 -2)\beta_k}
\]
is continuous and vanishes as $|(w_1, \dots, w_k)|$ tends to infinity.

 We have $H_0=1-\frac{1-H}{q}=1-\frac{1-H(m)}{m}$ with $H(m)$ as above. As a result, we obtain 
 the  convergence of the finite-dimensional distributions of  $T^{-1+(1-H_0)m}F_{n,q,\alpha, T}(t)$ to those the $m$th Hermite process 
$ Z^{H(m),m}$ multiplied by the constant $\frac {C_\alpha K_{x,\alpha,H_0}} { c_{H(m),m}} $. 

\bigskip
\noindent
\underline{\it Step 2}: {\it Tightness}.  Fix $0\le s<t \le 1$. To check that tightness holds in $C([0,1])$, let us  compute  the squared $L^2(\Omega)$-norm 
\[
\Phi_T :=T^{-1+(1-H_0)m} \E(| F_{n,q,\alpha, T}(t) -F_{n,q,\alpha, T}(s)|^2).
\]
Proceeding as in the first step of the proof, we obtain
\begin{align*}
\Psi_T &=\E \Bigg(  \Bigg| \int_{\R^m} dB(\xi_1) \cdots dB(\xi_m) \int_s^{t} du \int_{\R_+^n} dv_1 \cdots dv_n \prod_{k=1} ^n x(v_k) \\
&\times \prod_{1\le i<j \le n} | v_i-v_j|^{(2H_0-2)\alpha_{ij} }
\prod_{k=1}^n \prod_{\ell = 1+ \beta_{k-1}}^{\beta_k} (u-\frac{v_k}T - \xi_\ell)_+^{H_0 -\frac 32} \Bigg| ^2 \Bigg) \\
& \le m!   \int_{\R^m} d\xi_1 \cdots d\xi_m \Bigg| \int_s^{t} du \int_{\R_+^n} dv_1 \cdots dv_n \prod_{k=1} ^n x(v_k) \\
&\times \prod_{1\le i<j \le n} | v_i-v_j|^{(2H_0-2)\alpha_{ij} }
\prod_{k=1}^n \prod_{\ell = 1+ \beta_{k-1}}^{\beta_k} (u-\frac{v_k}T - \xi_\ell)_+^{H_0 -\frac 32} \Bigg| ^2.
\end{align*}
Using  (\ref{bienutile}) yields
\begin{align*}
\Psi_T & \le m!
\int_{[s,t]^2} du du' \int_{\R_+^{2n}} dv_1 \cdots dv_n  dv'_1 \cdots dv'_n   \\
&\times \prod_{k=1} ^n x(v_k) x(v'_k) 
  \prod_{1\le i<j \le n} | v_i-v_j|^{(2H_0-2)\alpha_{ij} } | v'_i-v'_j|^{(2H_0-2)\alpha_{ij} } \\
 &\times \prod_{k=1}^n   \beta(H_0 -\frac 12, 2-2H_0)^{\beta_k} |u-u' -\frac{v_k-v'_k}T |^{(2H_0 -2)\beta_k} \\
  & \le m! (t-s)
\int_{-1}^1 d\xi \int_{\R_+^{2n}} dv_1 \cdots dv_n  dv'_1 \cdots dv'_n   \\
&\times \prod_{k=1} ^n x(v_k) x(v'_k) 
  \prod_{1\le i<j \le n} | v_i-v_j|^{(2H_0-2)\alpha_{ij} } | v'_i-v'_j|^{(2H_0-2)\alpha_{ij} } \\
 &\times \prod_{k=1}^n   \beta(H_0 -\frac 12, 2-2H_0)^{\beta_k} |\xi-\frac{v_k-v'_k}T |^{(2H_0 -2)\beta_k} \\
 & \le C (t-s).
\end{align*}
Then the equivalence of all $L^p(\Omega)$-norms, $p\ge 2$,  on a fixed Wiener chaos, also known as the hypercontractivity property, allows us to conclude the proof of the tightness.
 \end{proof}

We will make use of  the notation
\begin{equation} \label{alpha0}
\alpha_0=(1-2H_0){\bf 1}_{\{\mbox{$nq$ is even}\}}-H_0{\bf 1}_{\{\mbox{$nq$ is odd}\}}.
\end{equation}

\begin{prp}\label{prp5}
Fix $n,q\geq 2$ and $\alpha\in A_{n,q}$, assume that the function $x$ belongs to $\mathcal S_L$ for some $L>1$ and that $ m\geq 3$. Then  for any $t\in [0,1]$, 
  $ T^{\alpha_0}F_{n,q,\alpha, T}(t)$ converge in $L^2(\Omega)$  to zero as $T\rightarrow \infty$; furthermore, the family
$\{ (F_{n,q,\alpha, T}(t))_{t\in [0,1]} , T>0\} $ is tight in $C([0,1])$.
\end{prp}

\begin{proof}
If $  (2H_0-2)m >-1$, by Proposition \ref{assympofF}, we know that
$T^{-1+m(1-H_0)}  F_{n,q,\alpha, T}(t)$ converges to zero in $L^2(\Omega)$ as $T\rightarrow \infty$. 
This implies the convergence to zero   in $L^2(\Omega)$ as $T\rightarrow \infty$ of  $ T^{-\alpha_0}F_{n,q,\alpha, T}(t)$  because
$-1+m(1-H_0) > \alpha_0$.
 We should then concentrate on the case $  (2H_0-2)m \leq-1$.
Once again, we shall divide the proof in two steps:

\bigskip 
\noindent
 \underline{\textit{Step 1}}: Let us first prove the convergence in $L^2(\Omega)$.
Fix $\alpha\in A_{n,q}$. We are going to show that 
\[
 \lim_{T\rightarrow \infty} T^{2\alpha_0} \E \left( | F_{n,q,\alpha, T} (t) |^2 \right) =0.
 \]
 We know that
\[
 T^{2\alpha_0} \E \left( | F_{n,q,\alpha, T} (t) |^2 \right) = T^{2\alpha_0}  m! 
 \times  \left\|  \int_{[0,Tt]} ds \otimes_\alpha (
g(s,\cdot) , \dots, g(s, \cdot))  \right\|^2_{L^2(\mathbb{R}^{ m})}.
\]
In view of the expression  for the contractions obtained in Proposition \ref{assympofF}, it suffices to show that  
  \begin{align*}
 & \lim_{T\rightarrow \infty} T^{2\alpha_0}   \int_{[0,Tt]^2} ds ds'  \int_{\mathbb{R}^{ m}}   \int_{(-\infty,s]^n} \int_{(-\infty,s']^n}
  dv_1 \cdots dv_n dv'_1 \cdots dv'_n d\xi_1 \cdots d\xi_{m}\\
  &\times \prod_{k=1}^n x(s-v_k) x(s'-v'_k) \prod_{1 \le i<j\le n} |v_i-v_j|^{(2H_0-2)\alpha_{ij}}|v'_i-v'_j|^{(2H_0-2)\alpha_{ij}} \\
  &\times \prod_{k=1}^n  \prod  _{\ell=1+  \beta_{j-1}} ^{ \beta_j} (v_k -\xi_{\ell})_+^{H_0-\frac 32}  \prod  _{\ell=1+   \beta_{j-1}} ^{ \beta_j} (v'_k -\xi_{\ell})_+^{H_0-\frac 32}=0.
  \end{align*}
  Integrating  in the variables $\xi$'s  and using \eqref{bienutile}, it remains to show that
   \begin{align*}
 & \lim_{T\rightarrow \infty} T^{2\alpha_0}   \int_{[0,Tt]^2} ds ds'      \int_{(-\infty,s]^n} \int_{(-\infty,s']^n}
  dv_1 \cdots dv_n dv'_1 \cdots dv'_n \\
  &\times  \prod_{k=1}^n x(s-v_k) x(s'-v'_k)  \prod_{1 \le i<j\le n} |v_i-v_j|^{(2H_0-2)\alpha_{ij}}|v'_i-v'_j|^{(2H_0-2)\alpha_{ij}}   \\ &\times \prod_{k=1}^n   |v_k-v'_k|^{(2H_0-2)\beta_k}=0.
  \end{align*} 
  Set
  \begin{align*}
 \Phi_T & :=  T^{2\alpha_0}   \int_{[0,Tt]^2} ds ds'      \int_{(-\infty,s]^n} \int_{(-\infty,s']^n}
  dv_1 \cdots dv_n dv'_1 \cdots dv'_n \\
  &\times   \prod_{k=1}^n x(s-v_k) x(s'-v'_k)  \prod_{1 \le i<j\le n} |v_i-v_j|^{(2H_0-2)
  \alpha_{ij}}|v'_i-v'_j|^{(2H_0-2)\alpha_{ij}} \\
  &\times \prod_{k=1}^n   |v_k-v'_k|^{(2H_0-2)\beta_k}.
  \end{align*}
  Making the change of variables $w_k=s-v_k$, $w'_k=s'-v'_k$ for $k=1,\dots,n$, yields
    \begin{align*}
 \Phi_T & =  T^{2\alpha_0}   \int_{[0,Tt]^2} ds ds'      \int_{\R_+^n} \int_{[0,\infty)^n}
  dw_1 \cdots dw_n dw'_1 \cdots dw'_n \\
  &\times  \prod_{k=1}^n x(w_k) x(w'_k)  \prod_{1 \le i<j\le n} |w_i-w_j|^{(2H_0-2)\alpha_{ij}}
  |w'_i-w'_j|^{(2H_0-2)\alpha_{ij}}  \\
  &\times\prod_{k=1}^n   |s-s'-w_k+w'_k|^{(2H_0-2)\beta_k}.
  \end{align*}
  Now we use Fubini's theorem and  make  the change of variables $s-s'=\xi$   to obtain
      \begin{align*}
 \Phi_T &  =  tT^{2\alpha_0+1}   \int_{\R_+^{2n}}       
  dw_1 \cdots dw_n dw'_1 \cdots dw'_n \\
  &\times  \prod_{k=1}^n |x(w_k) x(w'_k) | \prod_{1 \le i<j\le n} |w_i-w_j|^{(2H_0-2)\alpha_{ij}}
  |w'_i-w'_j|^{(2H_0-2)\alpha_{ij}}  \\
  &\times  \int_{-tT }^{tT}  d\xi  \prod_{k=1}^n   |\xi-w_k+w'_k|^{(2H_0-2)\beta_k}.
  \end{align*}
    
   We shall distinguish again two subcases:

\bigskip

  \noindent
     {\it Case $  (2H_0-2)m < -1$}:
     Notice that the exponent $2\alpha_0+1$ is negative:
     \begin{itemize}
 \item[(i)] If $nq$ is even, then $\alpha _0 =1-2H_0$ and
 \[
 2\alpha_0+1 =3-4H_0 <0
 \]
 because $H_0 >\frac 34$.  
  \item[(ii)]  If $nq$ is odd, then $\alpha _0 =-H_0$ and
 \[
 2\alpha_0+1=1-2H_0<0.
 \]
     \end{itemize}
     Therefore,  in order to show that $ \lim_{T\rightarrow \infty} \Phi_T =0$, it suffices to check that
   \begin{align}
  &  J:= \int_{\R^{2n}}         \notag
  dw_1 \cdots dw_n dw'_1 \cdots dw'_n   \prod_{k=1}^n |x(w_k) x(w'_k)|  \\  \notag
  & \qquad  \times  \prod_{1 \le i<j\le n} |w_i-w_j|^{(2H_0-2)\alpha _{ij}}
  |w'_i-w'_j|^{(2H_0-2)\alpha_{ij}}  \\  \label{ecu1}
  & \qquad \times  \int_{\mathbb{R}}  d\xi  \prod_{k=1}^n   |\xi -w_k+w'_k|^{(2H_0-2)\beta_k} <\infty,
  \end{align}
  where, by convention $x(w)=0$ if $w<0$.
We will apply the Power Counting Theorem  \ref{powercounting} to prove that this integral is finite.
We consider functions on $\R^{2n+1}$ with variables
 $\{(w_k)_{k\leq n},(w'_k)_{k\leq n},\xi\}$.
 The set of linear functions is
\begin{align*}
T&= \{\omega_k,  \omega'_k, 1\le  k\leq n\} \cup\{ w_i- w_j, w'_i-w'_j , 1\leq i<j\leq n\}\\
&\cup\{\xi-w_k+w'_k, 1\leq k\leq n\}.
\end{align*}
The corresponding exponents $(\mu_M,\nu_M)$ for each
  $M\in T$ are $(0,-L)$ for the  linear  functions $w_k$ and $w'_k$ (taking into account that $x\in \mathcal{S}_L$), 
  $ (2H_0-2) \alpha_{ij}$ for each function of the form $w_i- w_j$ or $ w'_i-w'_j$  and $(2H_0-2)\beta_k$ for each function of the form
 $ \xi-w_k+w'_k$.

  Then  $J<\infty$, provided conditions (a) and (b)  are satisfied.
 
\medskip
\noindent
$\bullet$ \textit{Verification of (b)}: Let $W\subset T$ be a linearly independent proper subset of $T$, and
\[
 d_\infty=2n+1-{\rm dim} ({\rm Span}(W))+\sum_{M\in T\setminus ({\rm Span}(W)\cap T)}\nu_M.
 \]
Let $S$ be the following subset of $T$: $S=\{ w_k, w'_k, 1\le k\le n\}$. Let $e={\rm Card}(S\cap {\rm Span}(W))$.
Consider the following two cases:
\begin{itemize}
  \item[(i)]  There exists $  k\leq n$ such that $\xi-w_k+w_k'\in {\rm Span}(W)\cap T$. Then  ${\rm dim}({\rm Span}(W))\geq e+1$. As a consequence, 
\begin{align*}
d_{\infty}\leq 2n+1-(e+1)-(2n-e)L<0,
\end{align*}
because $L>1$ and in this case, we should have $e<2n$ because $W$ is a proper subset of $T$.
\item[(ii)] Otherwise, 
\begin{align*}
d_{\infty}\leq 2n+1-e-(2n -e)L+(2H_0-2)m<0,
\end{align*}
because $L>1$ and $(2H_0-2)m<-1$.
\end{itemize}

$\bullet$ \textit{Verification of (a)}: A direct verification would require to solve a seemingly difficult combinatorial problem. We can simply remark that
\medskip
\begin{align*}
  &  \int_{[-1,1]^{2n}}         
  dw_1 \cdots dw_n dw'_1 \cdots dw'_n   \\ 
  & \qquad  \times  \prod_{1 \le i<j\le n} |w_i-w_j|^{(2H_0-2)\alpha _{ij}}
  |w'_i-w'_j|^{(2H_0-2)\alpha_{ij}}  \\  
  &\qquad\times  \int_{-1}^1  d\xi  \prod_{k=1}^n   |\xi -w_k+w'_k|^{(2H_0-2)\beta_k}\\
 =&m!\frac{1}{\beta(H_0-\frac12,2-2H_0)^{|\alpha |} }\mathbb{E}\left[\left(\int_0^1I^B_{nq-2|\alpha|}(f(s,\cdot),\ldots,f(s,\cdot))\right)^2\right]<\infty
\end{align*}
where $f(s,\xi_1,\ldots \xi_q)=\int_{-\infty}^{+\infty}\mathbb{I}_{[-1,1]}(s-v)\prod_{j=1}^q(v-\xi_j)_+^{H_0-3/2}dv$. Since
$x\in\mathcal{S}_L$, $x$ is bounded on $[-1,1]$. This implies that (a) is verified by the converse side of the Power Counting Theorem.

\medskip
\noindent
     {\it Case $  (2H_0-2)m =-1$}:
     In this case, we can apply  H\"older and Jensen inequalities to $\Phi_T$ in order to get
\begin{align*}
\Phi_T\leq T^{2\alpha_0+1 }A^\frac{\epsilon}{1+\epsilon}B^\frac{1}{1+\epsilon},
\end{align*}
with $2\alpha_0+1<0$,  $A=\left( \int_{\R} |x(w)| dw \right)^{2n}$ and
\begin{align*}
 B&=\int_{\mathbb{R}^{2n}}    
  dw_1 \cdots dw_n dw'_1 \cdots dw'_n   \prod_{k=1}^n |x(w_k) x(w'_k)|  \\  \notag
  & \qquad  \times  \prod_{1 \le i<j\le n} |w_i-w_j|^{(2H'_0-2)\alpha_{ij}}
  |w'_i-w'_j|^{(2H'_0-2)\alpha_{ij}}  \\  \label{ecu1}
  & \qquad \times  \int_{\mathbb{R}}  d\xi  \prod_{k=1}^n   |\xi -w_k+w'_k|^{(2H'_0-2)\beta_k},
\end{align*}
where $H'_0= H_0( 1+\epsilon) -\epsilon$. If $\epsilon$ is small enough, $H'_0$   can still be expressed as $1-\frac{1-H'}{q}$ for some $\frac{1}{2}<H'<H$. Moreover, in this case $(2H'_0-2)m<-1$ so we are exactly in the situation of the previous case, and the integral $B$ is finite.

\medskip
\noindent
  \underline{\textit{Step 2}}: 
Using the same arguments as previously and the hypercontractivity property, we deduce that there exists a constant $K>0$ such that
for all $0\le s<t \le 1$,
\begin{equation*}
 \mathbb{E}\left( | F_{n,q,\alpha, T} (t)-F_{n,q,\alpha, T} (s) |^4 \right)\leq K|t-s|^2,
\end{equation*}
which proves the tightness in $C([0,1])$.

\end{proof}

It remains to study what happens when $n=1$.
The proof of Proposition \ref{n=1} is very similar to 
that of Proposition \ref{assympofF} (although much simpler)   and details are left to the reader.

\begin{prp}\label{n=1}
Fix $q\geq 1$ and assume the function $x$ belongs to $L^1(\mathbb{R}_+)\cap L^{\frac1H}(\mathbb{R}_+)$.
Then   the finite-dimensional distributions of the process
\begin{equation}\label{paris}
G_T(t):=   T^{q(1-H_0)-1} \int_0^{Tt} ds I^B_q (g(s,\cdot)), \quad t\in [0,1]
\end{equation}
  converge in law  to those of a  $q$th Hermite process of Hurst parameter $1-q(1-H_0)$ multiplied by the constant
  $c_{H_0,q}^{-1} \int_0^\infty x(w) dw$,
and the family   \newline $\{(G_T(t))_{t\in [0,1]},\,T>0\}$ is tight in $C([0,1])$.
\end{prp}

We are now ready to make the proof of Theorem \ref{main}.

\medskip

\noindent
{\it Proof of Theorem \ref{main}}. It suffices to consider the decomposition
(\ref{repr}) and to apply the results shown in Propositions  \ref{assympofF} and   \ref{prp5}.
\qed

\section{Proof of Theorem \ref{FBM}}\label{sec4}

Let $Z$ be a fractional Brownian motion of Hurst index $H\in (\frac12, 1)$, and
let $x\in L^1(\R_+)\cap L^{\frac{1}{H}}(\R_+)$.
Consider the moving average process $X$ defined by 
$$
X(t)=\int_{-\infty}^t x(t-u)dZ_u,\quad t\geq 0,
$$
which is easily checked to be a stationary centered Gaussian process.
Denote by $\rho:\R\to\R$ the correlation function of $X$, that is, $\rho(t-s)=\E[X(t)X(s)]$, $s,t\geq 0$.
By multiplying the function $x$ by a constant if necessary, we can assume without loss of generality that
$\rho(0)=1$($={\rm Var}(X(t))$ for all $t$).
Let $P(x)=\sum_{n=0}^N a_nx^n$ be a real-valued polynomial function, and let $d$ denotes its centered Hermite rank.

\subsection{Proof of (\ref{cv2})}
In this section, we assume that $d\geq 1$  and that $H\in(1-\frac1{2d},1)$, and our goal is to show that
$$
T^{d(1-H)-1}\left\{\int_0^{Tt} \big(P(X(s))-\mathbb{E}[P(X(s))]\big)ds\right\}_{t\in [0,1]}
$$
converges in distribution in $C([0,1])$ to 
a Hermite process of index $d$ and Hurst parameter $1-d(1-H)$, up to some multiplicative constant $C_2$.
Since $P$ has centered Hermite rank $d$, it can be rewritten as
$$
P(x)=\E[P(X(s)]+\sum_{l=d}^N b_lH_l(x),
$$
for some $b_d,\ldots,b_N\in \R$, with $b_d\neq 0$ and $H_l$ the $l$th Hermite polynomial.
As a result,
we have
\begin{eqnarray*}
\int_0^{Tt} \big(P(X(s))-\mathbb{E}[P(X(s))]\big)ds &=& \sum_{l=d}^N b_l (c_{H,1})^l \int_0^{Tt} I_l^B(g(s,\cdot)^{\otimes l})ds,
\end{eqnarray*}
and the desired conclusion follows thanks to Propositions \ref{assympofF} and \ref{n=1}.

\subsection{Proof of (\ref{cv1})}
In this section, we assume that $d\geq 2$ and that $H\in(\frac12,1-\frac1{2d})$, and our goal is to show that
$$
T^{-\frac12}\left\{\int_0^{Tt} \big(P(X(s))-\mathbb{E}[P(X(s))]\big)ds\right\}_{t\in [0,1]}
$$
converges in distribution in $C([0,1])$ to 
a standard Brownian motion $W$, up to some multiplicative constant $C_1$.
To do so, we will rely on the Breuer-Major theorem, which asserts that the desired conclusion holds as soon as 
\begin{equation}\label{tocheck}
\int_\R |\rho(s)|^d ds<\infty
\end{equation}
(see, e.g., \cite{CampeseNourdinNualart} for a continuous version of the Breuer-Major theorem).

The rest of this section   is devoted to checking  that (\ref{tocheck}) holds true. Let us first compute $\rho$:
\begin{eqnarray*}
&&\rho(t-s)=\E[X(t)X(s)]\\
&=&H(2H-1)
\iint_{\R^2}x(t-v){\bf 1}_{(-\infty,t]}(v)x(s-u){\bf 1}_{(-\infty,s]}(u)|v-u|^{2H-2}dudv\\
&=&H(2H-1)
\iint_{\R^2}x(u)x(v)|t-s-v+u|^{2H-2}dudv,
\end{eqnarray*}
with the convention that $x(u)=0$ if $u<0$.  This allows us to write
\[
\rho(s) =c_H [ \tilde{x} * (I^{2H-1} x)](s),
\]
where $\tilde {x} (u)= x(-u)$, $I^{2H-1}$  is the fractional integral operator of order $2H-1$ and $c_H$ is a constant depending on $H$.
As a consequence, applying Young's inequality and Hardy-Littlewood's inequality (see \cite[Theorem 1]{Stein}) yields
\[
\| \rho \|_{L^d(\R)} \le c_{H} \| x\|_{L^p(\R)} \|I^{2H-1} x \| _{L^q(\R)} \le c_{H,p}  \| x\|^2_{L^p(\R)},
\]
where $\frac 1d = \frac 1p +\frac 1q -1$ and $\frac 1q = \frac 1p - (2H-1)$.  This implies  $p= (H+\frac 1{2d} )^{-1}$ and we have  $\| x\|_{L^p(\R)} <\infty$,
because $p\in (1,\frac 1H)$ and   $x\in L^1(\R)\cap L^{\frac{1}{H}}(\R)$.
 The proof of (\ref{cv1}) is complete.
 \qed

 \section{The Stationary Hermite-Ornstein-Uhlenbeck process}\label{sec5bis}
We dedicate this section to the study of the extension of the Ornstein Uhlenbeck process to the case where the driving process is a Hermite process. To our knowledge, there is not much literature about this object. Among the few existing references, we mention \cite{slaoui2018limit} and \cite{nourdin2019statistical}. The special case in which the driving process is a fractional Brownian motion has been, in contrast, well studied, see for instance \cite{cheredito2003}.
In what follows, we will prove a first-order ergodic theorem for the stationary Hermite-Ornstein-Uhlenbeck process. Then, we will use Theorem $\ref{main}$ to study its second order fluctuations.
 
Let $\alpha>0$.  Consider the function $x(s) =e^{-\alpha s} \mathbb{I}_{s>0}$  and let $Z^{H,q}$ be a Hermite process of order $q\ge 1$ and Hurst index $H>\frac12$. Then $x\in S_L$ for all $L>0$, and we can define the stationary Hermite-Ornstein-Uhlenbeck process as:
\begin{equation}
\label{HOU}(U_t)_{t\geq 0}=\int_{-\infty}^tx(t-s)dZ^{H,q}_s.
\end{equation}

As its name suggests, this process is strongly stationary. 
 We then state the following general ergodic type result.
\begin{prp}\label{ergo}
Let $(u_t)_{t\geq 0}$ be a real valued process of the form  $u_t= I^B_q(f_t)$,  where  $f_t \in L^2_s(\R^q)$  for each $t\ge 0$.
 Assume that $u$ is strongly stationary, has integrable sample paths and  satisfies, for each $1\le r\le q$, 
 \[
  \|f_0\otimes_r f_{s}\|_{ L^2(\R^{2q-2r}) }\underset{s\rightarrow\infty}\longrightarrow 0.
  \]
  Then, for all measurable function such that $\mathbb{E}[|f(u_0)|]<+\infty$, 
\[
\frac{1}{T}\int_0^Tf(u_s)ds\overset{a.s}{\underset{T\rightarrow\infty}\longrightarrow}\mathbb{E}\left[f(u_0)\right].
\]
\end{prp}
\begin{proof}
According to Theorem 1.3 in \cite{nourdin2016strong}, the process $u$ is strongly mixing if 
for all $t>0$ and    $1\leq r\leq q$, the following convergence holds
\[
 \|f_t\otimes_r f_{t+s}\|_{L^2(\R^{2q-2r})}\underset{s\rightarrow\infty}\longrightarrow 0.
 \]
 Taking into account that $u$ is  strongly stationary, we can write 
 \[
 \|f_t\otimes_r f_{t+s}\|_{L^2(\R^{2q-2r})}=\|f_0\otimes_r f_{s}\|_{L^2(\R^{2q-2r})},
 \] and the conclusion follows immediately from Birkhoff's continuous ergodic theorem.
\end{proof}

We can now particularize to the Hermite-Ornstein-Uhlenbeck process.
\begin{thm}\label{firstOrder}
Let $U$ be the Hermite-Ornstein-Uhlenbeck process defined by \eqref{HOU}. Let $f$ be a measurable function such that 
$|f(x)|\leq \exp(|x|^\gamma)$ for some $\gamma<\frac{2}{q}$. Then,
\begin{equation}\lim_{T\rightarrow\infty}\frac{1}{T}\int_0^Tf(U_s)ds=\mathbb{E}\left[f(U_0)\right] \textit{ a.s}.\notag\end{equation}
\end{thm}

\begin{proof}
We shall prove that the process $U$ verifies the conditions of Proposition \ref{ergo}. We have $U_t=I_q^B(f_t)$ with 
\[
f_t(x_1,\dots,x_q)=c_{H,q}\mathbb{I}_{[-\infty,t]^q}(x_1,\dots,x_q)\int_{x_1\vee \cdots \vee x_q}^te^{-\alpha(t-u)}\prod_{i=1}^q(u-x_i)^{H_0-\frac{3}{2}}du.
\]

\medskip
\noindent
\underline{\it Step 1.}  Let us first show the mixing condition, that is 
\[
\lim_{s\rightarrow \infty} \|f_0\otimes_r f_{s}\|_{L^2(\R^{2q-2r})}=0
\]
 for all  $r\in\{1,\dots,q\}$.
We can write
\begin{align*}
 & f_0\otimes_rf_s(y_1,\dots, y_{2q-2r}) \\
 &=c^2_{H,q}\int_{(-\infty,0]^r}\left(\int_{x_1\vee \cdots \vee x_r\vee y_{1}\cdots\vee y_{q-r}}^0e^{\alpha u}\prod_{i=1}^r\prod_{j=1}^{q-r}(u-x_i)^{H_0-\frac{3}{2}}(u-y_j)^{H_0-\frac{3}{2}}du\right)\\
 &\quad\times\Bigg(\int_{x_1\vee \cdots \vee x_r\vee y_{q-r+1}\cdots\vee y_{2q-2r}}^{s}e^{-\alpha(s-u)}\\
 &\quad \times \prod_{i=1}^r\prod_{j=q-r+1}^{2q-2r}(u-x_i)^{H_0-\frac{3}{2}}(u-y_j)^{H_0-\frac{3}{2}}du\Bigg)dx_1  \cdots dx_r \\ 
 &=c^2_{H,q}\int_{y_{1}\vee\cdots \vee y_{q-r}}^0e^{\alpha u}\int_{y_{q-r+1}\vee\cdots\vee y_{2q-2r}}^{s}e^{-\alpha(s-v)}\\
 & \quad \times \left(\int_{(-\infty,u\wedge v]}(u-x)^{H_0-\frac{3}{2}}(v-x)^{H_0-\frac{3}{2}}dx\right)^r\\
 &\quad \times\prod_{j=1}^{q-r}\prod_{l=q-r+1}^{2q-2r}(u-y_j)^{H_0-\frac{3}{2}}(v-y_l)^{H_0-\frac{3}{2}}dvdu  \\
& =c^2_{H,q}\beta(H_0-\frac{1}{2}, 2-2H_0)^r\int_{y_{1}\vee \cdots \vee y_{q-r}}^0e^{\alpha u}\int_{y_{q-r+1}\vee \cdots \vee y_{2q-2r}}^{s}e^{-\alpha(s-v)}|u-v|^{r(2H_0-2)}\\
& \quad \times\prod_{j=1}^{q-r}\prod_{l=q-r+1}^{2q-2r}(u-y_j)^{H_0-\frac{3}{2}}(v-y_l)^{H_0-\frac{3}{2}}dvdu,
\end{align*}
where we used again the identity \eqref{bienutile}.
  We then have
\begin{align*}
& \|f_0\otimes_rf_{s}\|^2_{L^2(\mathbb{R}^{2q-2r})}= c^4_{H,q}\beta(H_0-\frac{1}{2}, 2-2H_0)^{2q}\\
& \quad \times \int_{(-\infty,0]^2} \int_{ (-\infty,s]^2}e^{\alpha(u+u_1)}e^{-\alpha(2s-(v+v_1))}|u-u_1|^{(q-r)(2H_0-2)}\\
& \quad \times|v-v_1|^{(q-r)(2H_0-2)}|u-v|^{r(2H_0-2)}|u_1-v_1|^{r(2H_0-2)}dv_1dvdu_1du\\ 
&\leq  c^4_{H,q}\beta(H_0-\frac{1}{2}, 2-2H_0)^{2q}A_0A_{s}R_{s}^2,
\end{align*}
with 
\[
 A_x=\left(\int_{(-\infty,x]^2}e^{-q\alpha(2x-(u+u_1))}|u-u_1|^{q(2H_0-2)}dudu_1\right)^\frac{1}{a}
 \]
 and
 \[
R_{s}=\left(\int_{-\infty}^0\int_{-\infty}^{s}e^{-q\alpha(s-(u+v))}|u-v|^{q(2H_0-2)}dvdu\right)^\frac{1}{b},
\]
where we used the H\"older inequality with $a=\frac{q}{q-r},b=\frac{q}{r}$.
 Making the change of variable $x-u=v$, $x-u_1=v_1$, we obtain
 \begin{align*}
& \int_{(-\infty,x]^2}e^{-q\alpha(2x-(u+u_1))}|u-u_1|^{q(2H_0-2)}dudu_1\\
& \quad =\int_{[0,\infty)^2}e^{-q\alpha( v+v_1)}|v-v_1|^{q(2H_0-2)}dudu_1<\infty.
\end{align*}
 On the other hand, we have $q(2H_0-2)=2H-2$, so

\begin{align*}
R^b_s={\rm Cov}(U^H_0,U^H_s)
\end{align*}
where $U^H$ is a stationary Ornstein Uhlenbeck process driven by a fractional Brownian motion of index $H$ (and with $\alpha^H=q\alpha$). According to \cite[Lemma 2.2]{cheredito2003}, one has $R^b_s=O_{s\rightarrow\infty}(s^{-H})$, implying in turn that
 $ \lim_{s\rightarrow\infty}R_{s}=0$
and concluding the proof of the mixing condition.
 
 \medskip
\noindent
\underline{\it Step 2.}  We now show the integrability condition $\E[ | f(U_0)|] <\infty$. From the results of   \cite{cheredito2003},
we have $\mathbb{E}[U_0^2]=\frac{1}{\alpha^{2H}}\Gamma(2H)$. A power series development yields
\[
\mathbb{E}[|f(U_0)|]\leq \sum_{k=0}^{\infty}\frac{1}{k!}\mathbb{E}[|U_0|^{\gamma k}],
\]
where $U_0$ is an element of the $q$th Wiener chaos. By the hypercontractivity property, for all $ k\geq \frac{2}{\gamma}$,
\[
\mathbb{E}[|U_0|^{\gamma k}]\leq g(k):= (k-1)^{\frac{\gamma qk}{2}}\left(\frac{1}{\alpha^{2H}}\Gamma(2H)\right)^\frac{\gamma k}{2}.
\]
 Stirling formula allows us to write
  \begin{equation}\frac{g(k)}{k!}\sim_{k\rightarrow\infty}\frac{(k-1)^{\frac{\gamma qk}{2}}}{k^{k}}\frac{\left(\frac{1}{\alpha^{2H}}\Gamma(2H)\right)^\frac{\gamma k}{2}e^k}{\sqrt{2\pi k}},\end{equation}
and the associated series converges if $\gamma q<2$.
\end{proof}

 The next result analyzes the fluctuations in the   ergodic theorem proved in Theorem  \ref{firstOrder}.
 
\begin{thm}\label{secondOrder}
\textit{(A)} {\rm [Case $q=1$]} Let $f$ be in $L^2(\R, \gamma)$ for $\gamma=\mathcal N(0,\frac{\Gamma(2H)}{\alpha^{2H}})$.
We denote by  $(a_i)_{i\geq 0}$ the coefficients of $f$ in its Hermite expansion, and we let $d$ be the centered Hermite rank of $f$. Then,
\begin{itemize}
  \item if $\frac{1}{2}<H<1-\frac{1}{2d}$,
\begin{equation*}
\frac{1}{\sqrt{T}}\int_0^{Tt}\left(f(U_s)-\mathbb{E}[f(U_0)]\right)ds\overset{f.d.d}{\underset{T\rightarrow\infty}\longrightarrow} c_{f,H}W_t,
\end{equation*}
   \item if $H=1-\frac{1}{2d}$,
\begin{equation*}
\frac{1}{\sqrt{T\log{T}}}\int_0^{Tt}\left(f(U_s)-\mathbb{E}[f(U_0)]\right)ds\overset{f.d.d}{\underset{T\rightarrow\infty}\longrightarrow}c_{f,H}W_t,
\end{equation*}
   \item if $H>1-\frac{1}{2d}$,
\begin{equation*}
T^{q(1-H)-1}\int_0^{Tt}\left(f(U_s)-\mathbb{E}[f(U_0)]\right)ds\overset{f.d.d}{\underset{T\rightarrow\infty}\longrightarrow} c_{f,H}Z^{d,H}_t,
\end{equation*}
\end{itemize}
where $Z^{d,H}$ is a Hermite process of order $d$ and index $d(H-1)+1$, $W$ is a Brownian motion and 
\begin{equation}
c_{f,H}
=
\left\{
\begin{array}{cl}
\sqrt{\sum_{k\geq d}k!a^2_k\int_{\mathbb{R}_+}|\rho(s)|^k} & \textit{ if } \,\, H<1-\frac{1}{2d}\\
a_d\sqrt{d!\frac{3}{16\alpha^2}}      & \textit{ if } \,\, H=1-\frac{1}{2d}\\
a_d\sqrt{d!\frac{H^d\Gamma(2H)^d}{\alpha^{2Hd}}}      & \textit{ if } \,\, H>1-\frac{1}{2d}
\end{array}
\right.
\end{equation}
with 
\[
\rho(s)=\mathbb{E}[U_sU_0]=\int_{-\infty}^0\int_{-\infty}^se^{-\alpha(s-(u+v))}|u-v|^{2H-2}dudv.
\]
Moreover, if $f\in L^p(\R,\gamma)$ for some $p>2$, the previous convergences holds true in the Banach space $C([0,1])$.\\
\textit{(B)} {\rm [Case $q>1$]} Let $P$ be a real valued polynomial. Then, the conclusions of Theorem \ref{main} apply to $U$.
\end{thm}
\begin{proof}
Except for $H=1-\frac{1}{2d}$ in Part A, this is a direct consequence of Theorems \ref{FBM} and \ref{main}.
The convergence in the critical case 
can be checked through easy but tedious computations, by reducing to the case where $f$ is the $d$th Hermite polynomial. Details are left to the reader.
\end{proof}

\section{Appendix}\label{sec6}
In this section we present two technical lemmas that  play an important  role along the paper.
First, we shall reproduce a very useful result from \cite{Taqqu2}:

\begin{thm}[Power Counting Theorem]\label{powercounting}
Let $T=\{M_{1},\ldots,M_K\}$ a set of linear functionals on $\mathbb{R}^n$, $\{f_1,\ldots, f_K\}$ a set of real measurable functions
on $\R^n$ such that  there exist  real numbers $(a_i,b_i,\mu_i,\nu_i)_{1\leq i\leq K}$,  satisfying for each $i=1, \dots, K$,
\begin{eqnarray*}
&& 0<a_i\leq b_i,\\
&& |f_i(x)| \le |x|^{\mu_i} \textit{ if } |x|\leq a_i,\\
&& |f_i(x)| \le  |x|^{\nu_i} \textit{ if } |x|\geq b_i,\\
&&  f_i \textit{ is bounded over } [a_i,b_i].
\end{eqnarray*}
For   a linearly independent subset of $W$ of  $T$, we write $S_T(W)={\rm Span}(M)\cap T$. We also define
\begin{eqnarray*}
d_0(W)={\rm dim}({\rm Span}(W))+\sum_{i: M_i \in S_T(W)}\mu_i, \\
d_\infty(W)=n-{\rm dim}({\rm Span}(W))+\sum_{i:M_i \in T\setminus S_T(W)}\nu_i.
\end{eqnarray*}
Assume ${\rm dim}({\rm Span}(T))=n$. Then, the two conditions
$(a):d_0(W)>0$ for all linearly independent subsets $W\subset T$, $(b):d_{\infty}(W')<0$ for all linearly independent proper subsets $W'\subset T$,  imply
\begin{equation}  \label{finite}
\int_{\mathbb{R}^n}\prod_{i=1}^K|f_i(M_i(x))|dx<\infty 
\end{equation}
Moreover, assume that $|f_i(x)| = |x|^{\mu_i} \textit{ if } |x|\leq a_i$, Then\\
 $ \int_{[-1,1]^n}\prod_{i=1}^K|f_i(M_i(x))|dx<\infty$,  if an only if  
   for any linearly independent subset $W\subset T$ condition $(a)$ holds.
\end{thm}

The next lemma is an application of the Hardy-Littlewood-Sobolev inequality,
\begin{lem}\label{HLS-lemme}
Fix $n,q\geq 2$ and $\alpha\in A_{n,q}$. Recall    $H$ and $H_0$ from (\ref{hzero}). 
Assume $x\in L^1(\mathbb{R})\cap L^{\frac1H}(\mathbb{R})$
Then
$$\int_{\mathbb{R}^n} \prod_{k=1}^n|x(\eta_k)| \prod_{1\leq i<j\leq n}|\eta_i-\eta_j|^{(2H_0-2)\alpha_{ij}}  d\eta_1\ldots d\eta_n<\infty.$$ 
\end{lem}
\begin{proof}
We are going to use
 the multilinear Hardy-Littlewood-Sobolev inequality, that we recall here for the convenience of the reader (see \cite[Theorem 6]{Beckner}):
if $f:\mathbb{R}\to\mathbb{R}$ is a measurable function, if $p\in(1,n)$ and if the $\gamma_{ij}\in(0,1)$ are such that $\sum_{1\leq i<j\leq n}\gamma_{ij}=1-\frac 1p$, then there exists $c_{p,\gamma}>0$ such that
\begin{equation}\label{HLS}
\int_{\mathbb{R}^n} \prod_{k=1}^n|f(u_k)|  \prod_{1\leq i<j\leq n}|u_i-u_j|^{-\gamma_{ij}}du_1\ldots du_n\leq c_{p,\gamma} \left(\int_{\mathbb{R}}|f(u)|^pdu\right)^{\frac{n}p}.
\end{equation}

Set $p=1/(1-(1-H)\frac{2|\alpha|}{nq})$.
Since $2|\alpha|\leq nq$, we have that $p>1$. 
On the other hand, since $H>\frac12$, one has $nH>\frac{n}2\geq 1$; this implies that
$(1-H)\frac{2|\alpha|}{q}<(1-H)n<n-1$, that is, $p<n$.
Moreover, set $\gamma_{ij}=(2-2H_0)\alpha_{ij}=(1-H)\frac{2\alpha_{ij}}{q}\in(0,1)$;
we have 
$
\sum_{1\leq i<j\leq n}\gamma_{ij}=2(1-H)\frac{|\alpha|}{q}\leq (1-H)n<n-1.
$
We deduce from (\ref{HLS}) that
$$
\int_{\mathbb{R}^n} \prod_{k=1}^n|x(\eta_k)| \prod_{1\leq i<j\leq n}|\eta_i-\eta_j|^{(2H_0-2)\alpha_{ij}}  d\eta_1\ldots d\eta_n\leq c_{p,\gamma}
\left(\int_{-\infty}^\infty|x(u)|^pdu\right)^{\frac{n}{p}}.
$$
But $p\in(1,\frac1H)$ and $x\in L^1(\mathbb{R})\cap L^{\frac1H}(\mathbb{R})$, so the claim follows.
\end{proof}

\medskip
\noindent
{\bf Acknowledgments}: We are grateful to Pawel Wolff for making us aware
of the multivariate Hardy-Littlewood-Sobolev inequality from \cite{Beckner}.

\end{document}